\documentclass[12pt]{article}
\usepackage{geometry}
\usepackage[tbtags]{mathtools}
\usepackage{amsthm,amssymb}
\usepackage{enumerate}
\usepackage{xcolor,graphicx}
\usepackage{tikz,tikz-cd}
\usepackage[colorlinks, pagebackref,
linkcolor=red, citecolor=blue
]{hyperref}
\usepackage{stix} % gives \opluslhrim  ⨭ and \oplusrhrim ⨮ commands
\usepackage{bbm} % gives \boldmath command for bold math symbols

\usepackage[scr=boondox,  % heavily sloped
bb=pazo, %bb=fourier,
cal=esstix]   % slightly sloped
{mathalpha}

\usepackage[backgroundcolor=yellow,textwidth=1.7in]{todonotes}

\theoremstyle{plain}
\newtheorem{theorem}{Theorem}[section]
\newtheorem{lemma}[theorem]{Lemma}
\newtheorem{corollary}[theorem]{Corollary}
\newtheorem{proposition}[theorem]{Proposition}
\newtheorem{conjecture}[theorem]{Conjecture}

\theoremstyle{definition}
\newtheorem{definition}[theorem]{Definition}
\newtheorem{example}[theorem]{\sc Example}

\newtheorem{remark}[theorem]{Remark}

\newcommand{\ds}{\displaystyle}
\newcommand{\col}{\!:\!}

\DeclareMathOperator{\ad}{ad}
\DeclareMathOperator{\Id}{Id}
\DeclareMathOperator{\Aut}{Aut}
\DeclareMathOperator{\End}{End}
\DeclareMathOperator{\Hom}{Hom}
\DeclareMathOperator{\Lie}{Lie}
\DeclareMathOperator{\Der}{Der}
\DeclareMathOperator{\iDer}{iDer}
\DeclareMathOperator{\oDer}{oDer}
\DeclareMathOperator{\aDer}{aDer}

\newcommand{\aut}[1]{\Aut(#1)}
\newcommand{\aaut}{\operatorname{Aut^a}}
\newcommand{\der}[1]{\Der(#1)}
\newcommand{\dero}[1]{\Der^\circ(#1)}
\newcommand{\ider}[1]{\iDer(#1)}
\newcommand{\oder}[1]{\oDer(#1)}
\newcommand{\ader}[1]{\aDer(#1)}
\newcommand{\idero}[1]{\iDer^\circ(#1)}
\newcommand{\adero}[1]{\aDer^\circ(#1)}

\newcommand{\ai}{almost inner\ }

\newcommand{\zer}{\mathbf{0}}
\newcommand{\one}{\mathbf{1}}

\newcommand{\Z}{\mathbb{Z}}

\newcommand{\R}{\mathbb{R}}
\newcommand{\C}{\mathbb{C}}
\newcommand{\F}{\mathbb{F}}

\newcommand{\Zt}{\Z_2}
\newcommand{\cE}{\mathcal{E}}
\newcommand{\cO}{\mathcal{O}}

\newcommand{\g}{\mathfrak{g}}
\newcommand{\gl}{\mathfrak{gl}}

\renewcommand{\sl}{\mathfrak{sl}}
\newcommand{\psl}{\mathfrak{psl}}
\newcommand{\so}{\mathfrak{so}}
\renewcommand{\sp}{\mathfrak{sp}}
\newcommand{\osp}{\mathfrak{osp}}

\newcommand{\psq}{\mathfrak{psq}}
\newcommand{\spe}{\mathfrak{spe}}

\title{Almost inner derivations of Lie superalgebras}

\author{Vera Serganova%
  \thanks{Department of Mathematics, University of California, Berkeley,
  \texttt{\href{mailto:serganov@berkeley.edu}{serganov@berkeley.edu}}
   }
\quad  and \quad
Arkady Vaintrob%
\thanks{Department of Mathematics, University of Oregon, Eugene, \ \texttt{\href{mailto:vaintrob@uoregon.edu} {vaintrob@uoregon.edu}}}
}

\date{}

\begin{document}

\maketitle
\begin{abstract}
An \ai derivation of a Lie algebra $L$ is a derivation that coincides with an inner derivation on each one-dimensional subspace of $L$. The almost inner derivations form a subalgebra $\ader{L}$ of the Lie algebra $\der{L}$ of all derivations of $L$, containing the inner derivations $\ider{L}$ as an ideal. If $L$ is a simple finite-dimensional Lie algebra, then $\ader{L}=\ider{L}$, since all derivations of $L$  are inner.

In this paper, we introduce and study \ai derivations of Lie superalgebras. Since simple Lie superalgebras may admit non-inner outer derivations, the existence of non-inner \ai derivations becomes a nontrivial question. Nevertheless, we show that all \ai derivations of finite-dimensional simple Lie superalgebras over $\mathbb C$ are inner.
We also give examples of naturally occurring non-inner \ai derivations of some pseudo-reductive Lie superalgebras related to the Sato-Kimura classification of prehomogeneous vector spaces.
\end{abstract}

\section{Introduction}
Let $L$ be a Lie algebra. A derivation $D$ of $L$ is called \emph{\ai}if it coincides with an inner  derivation on each one-dimensional subspace of $L$, that is, for each $x\in L$, there exists $a\in L$ such that $D(x)=[a,x]$.
If $G$ is a Lie group with $\Lie(G)=L$, then such a derivation $D$ generates a one-parameter subgroup of automorphisms of $G$ that sends every element $g\in G$ to its conjugate. Such automorphisms are called \emph{\ai}or \emph{class-preserving}.
Ever since Burnside's discovery~\cite{Bu} of the existence of non-inner class-preserving automorphisms, such automorphisms have attracted attention of finite group theorists (see~\cite{Ku} or~\cite{Ya} for a recent review).

Almost inner derivations first appeared as a tool for producing non-inner class-preserving automorphisms in the work of Gordon and Wilson~\cite{GW} on isospectral deformations of Riemannian manifolds. (Later and independently, they were introduced in~\cite{SSB,AS} under the name  \emph{pointwise inner derivations}.)
Gordon and Wilson proved that if  $\phi:G\to G$ is an \ai automorphism of a solvable Lie group $G$ with a right-invariant metric $g$, then for any cocompact discrete subgroup  $\Gamma\subset G$, the Riemannian manifolds $X=G/\Gamma$ and $X'=G/\phi(\Gamma)$ are isospectral, that is, the Laplace operators on $\Omega^p(X)$ and $\Omega^p(X')$ have the same spectrum for all $p\ge 0$. If $\phi$ is inner, then $X$ and $X'$ are also isometric, but otherwise they may not be. This observation led Gordon and Wilson to the discovery of the first known example of a continuous family of isospectral, pairwise non-isometric Riemannian manifolds.

Since then, \ai automorphisms and derivations, along with related concepts, such as almost conjugate subgroups and homomorphisms, have appeared in various areas of algebra and its applications, especially where the interplay between local and global conjugacy plays a role (see e.g.\ \cite{FHS,Go,KO,La,O,Wa,We}).

A systematic algebraic study of \ai derivations of Lie algebras was initiated in~\cite{BDV} (see also~\cite{AS}). For finite-dimensional simple Lie algebras, all derivations are inner and therefore they have no nontrivial \ai derivations. Thus it is not surprising that all known examples of \ai derivations involve nilpotent or solvable Lie algebras (see \cite{BDV2,Ve,DG}). On the other hand, many simple Lie superalgebras have outer derivations, which raises the question of whether they may have nontrivial \ai derivations.

In this note, we show that, nevertheless, all \ai derivations of finite-dimensional simple Lie superalgebras over the field of complex numbers are inner.
However, we also found examples of non-inner \ai derivations for some non-simple Lie superalgebras that are far from being solvable. In particular, such derivations exist for several families of \emph{quasireductive} Lie superalgebras (i.e.\ superalgebras $L=L_\zer\oplus L_\one$, where $L_\zer$ is a reductive Lie algebra and $L_\one$ is a semisimple $L_\zer$-module, see~\cite{Se}). Quite unexpectedly, these examples turned out to be related with prehomogeneous vector spaces~\cite{Vi,SK}.

Let us describe the contents of the paper in more detail. We begin in Section~\ref{sec:almost} with the definition and general properties of \ai derivations for Lie algebras.
Also, we prove there that for a finite-dimensional Lie algebra $L$ over a field of characteristic zero, the \ai derivations form an ideal in the Lie algebra $\der{L}$ of all derivations. This fact was previously an open conjecture~\cite{BDV}, verified only in some special cases such as nilpotent Lie algebras~\cite{KO}.

Next, we define \ai derivations of Lie superalgebras as derivations that coincide with inner derivations on elementary subsuperalgebras. We also give several natural examples of Lie superalgebras with nontrivial \ai derivations, both even and odd, and show how such examples can be constructed from certain prehomogeneous vector spaces for semisimple Lie algebras.

In Section~\ref{sec:classification}, we review the classification and structure of simple Lie superalgebras given by V.\,Kac~\cite{Ka} and provide a complete description of their superalgebras of derivations.
In Section~\ref{sec:main}, we prove our main result (Theorem~\ref{thm:main}) that all \ai derivations of simple Lie superalgebras over $\C$ are inner.
In Section~\ref{sec:quasi}, we study \ai derivations of quasireductive Lie superalgebras.
Finally, in Section~\ref{sec:final}, we present several additional examples and discuss directions for future work.

\

We thank Boris Kunyavskii, who first told us about almost inner derivations and raised the question of their existence for simple Lie superalgebras.

\section{Almost inner derivations}
\label{sec:almost}

\subsection{Preliminaries}
We work with vector spaces, linear maps, algebras, etc., defined over a base field $\F$ of characteristic not equal to $2$.
A vector \emph{superspace} is the same as a $\Zt$-graded vector space
$$V = V_\zer\oplus V_\one.$$
Elements of $V_\zer$ are called \emph{even} and elements of $V_\one$ are called \emph{odd}. For $v\in V_\alpha$, we call $\alpha \in \Zt=\{\one,\zer\}$ the \emph{degree} (or \emph{parity}) of $v$ and denote it by $|v|$.

We say that a linear map
$$f:V\to W$$ between vector superspaces
$V=V_\zer\oplus V_\one$ and $W=W_\zer\oplus W_\one$ has degree $\delta\in \Zt$ if $f(V_\alpha)\subset W_{\alpha+\delta}$ for $\alpha\in \Zt$. This introduces a $\Zt$-grading on the space of $\F$-linear maps
$$
\Hom(V,W)=\Hom(V,W)_\zer\oplus \Hom(V,W)_\one.
$$
In dealing with superspaces, formulas are often written only for homogeneous elements and then extended to inhomogeneous ones by linearity.
In all operations involving parity, such as tensor products, commutators, and permutations of variables, we follow the standard \emph{sign rule}:
\begin{quote}
  whenever two homogeneous elements $u$ and $v$ are interchanged,  a sign factor $(-1)^{|u|\cdot |v|}$ is introduced.
\end{quote}
For example, the \emph{commutator} of two endomorphisms
$$f,g\in \End_\F(V)=\Hom_\F(V,V)$$
of a superspace $V$ is defined as
\begin{equation}
  \label{eq:commutator}
  [f,g]=fg - (-1)^{|f| \cdot |g|}gf.
\end{equation}

\subsection{Almost inner derivations of Lie algebras}
Recall that a derivation of a Lie algebra $L$ over a field $\F$ is a linear map
$$ D:L\to L $$
satisfying the Leibniz identity
\[
  D([x,y])=[D(x),y]+[x,D(y)], \quad  \text{for all } x,y\in L.
\]
It is easy to check that the set $\der{L}$ of all derivations of $L$ is a Lie subalgebra of the Lie algebra $\End_\F(L)$.

The Jacobi identity for $L$ can be restated as the condition that for all $a\in L$, the map
\[
  \ad_a: L\to L, \quad ad_a(x)=[a,x]
\]
is a derivation. The map
$$\ad: L \to \der{L}$$
is a Lie algebra homomorphism whose kernel
is the center $Z(L)$ of $L$.

Derivations of the type $\ad_a$ are called \emph{inner}, and it is easy to see that the set
\[
  \ider{L}=\{\ad_a \col a\in L\}\subset \der{L}
\]
  of all inner derivations is an ideal in $\der{L}$.

  \

 The quotient Lie algebra
 \begin{equation}
   \label{eq:outer}
\oder{L}=\der{L}/\ider{L}
\end{equation}
is called the algebra of \emph{outer derivations} of $L$.
By definition, it coincides with the first cohomology of $L$ with the coefficients in the adjoint module

\begin{equation}
  \label{eq:cohom}
  \oder{L}=H^1(L,L).
\end{equation}
The well-known fact that, if $L$ is a finite-dimensional semisimple Lie algebra and $M$ is an non-trivial irreducible finite-dimensional $L$-module, then
\begin{equation}
  \label{eq:vanish}
H^i(L,M)=0, \text{ for all } i\ge 0,
\end{equation}
implies that all derivations of a semisimple Lie algebra are inner.

\begin{definition}
\label{def:almost} Let $L$ be a Lie algebra.
A derivation $D\in \der{L}$ is called \emph{\ai}if
on every one-dimensional subspace of $L$, it coincides with an inner derivation.
In other words, $D$ is \ai if for every $x\in L$ there exists $a\in L$ such that
\begin{equation}
  \label{eq:ai}
 D(x)=[a,x],
\end{equation}
or equivalently, if
\begin{equation}
  \label{eq:ai2}
D(x)\in [L,x] \text{ \ for all } x\in L.
\end{equation}
We denote by
$$
\ader{L} \subset \der{L}
$$
the subspace of all \ai derivations of $L$.
\end{definition}

\begin{proposition}
  \label{prop:subalg}
The space $\ader{L}$  is a Lie subalgebra of $\der{L}$.
\end{proposition}
\begin{proof}
  Let $D,D'\in \ader{L}$. Then, for $x\in L$, we have $D(x)=[a,x]$
  and  $D'(x)=[a',x]$ for some $a,a'\in L$. Therefore
\[
 \begin{split}[D',D](x)=&D'(D(x))-D(D'(x))=D'([a,x])-D([a',x])\\
  =&[D'(a),x]+[a,D'(x)]-[D(a'),x]-[a',D(x)]\\
  =&[D'(a),x]+[a,[a',x]]-[D(a'),x]-[a',[a,x])]\\
  =&[D'(a),x]-[D(a'),x]+[[a,a'],x]\\
  =&[D'(a)-D(a')+[a,a'],x].
 \end{split}
\]
Thus $[D,D']\in \ader{L}$.
\end{proof}

\begin{remark}
  If $L=\mathrm{Lie}(G)$,  the Lie algebra of a simply connected Lie group $G$, then $\der{L}$ can be naturally identified with the Lie algebra of the Lie group $\aut{G}$ of automorphisms of $G$. Under this identification, $\ider{L}$ becomes the Lie algebra of the group of inner automorphism of $G$, and one can expects  $\ader{L}$ to be the Lie algebra of the subgroup  consisting of \emph{\ai}(also known as  \emph{class preserving}) automorphisms of $G$, i.e.\ automorphisms $\gamma \in \aut{G}$  such that for every $g\in G$ we have $\gamma(g) = h^{-1}gh$ for some $h\in G$ or, equivalently, of the subgroup of automorphisms preserving each adjoint orbit in $L$. The only sticky point is to show that \ai automorphisms is a Lie group.
\end{remark}

In fact, \ai derivations form an ideal in the Lie algebra of all derivations.
First, let us check that \ai derivations are preserved by automorphisms.

\begin{lemma} \label{lem:autom}
  Let $D\in \ader{L}$ be an \ai  derivation and let $f\in \aut{L}$ be an automorphism of $L$.
  Then
  $$f_*(D):=f\circ D\circ f^{-1}$$ is also an \ai derivation.
\end{lemma}
\begin{proof} Indeed, let $x\in L$. Since $D\in \ader{L}$, we have
  $$D(f^{-1}(x))=[y,f^{-1}(x)]$$  for some $y\in L$. Therefore,
$$
  f_*(D)(x)=fDf^{-1}(x)=f(D(f^{-1}(x)))=f([y,f^{-1}(x)])=[f(y),x],
$$
  which shows that $f_*(D)\in \ader{L}$.
\end{proof}

\begin{proposition}
\label{prop:ideal}
Let $L$ be a finite-dimensional Lie algebra over a field $\F$ of characteristic $0$. Then the space $\ader{L}$ is an ideal in $\der L$.
\end{proposition}

\begin{proof} Consider the group $\aut L$ of automorphisms of $L$. By  Lemma~\ref{lem:autom} the subspace $\ader L\subset\der L$ is stable under the action of $\aut L$. On the other hand, $\der L$ is the Lie algebra of $\aut L$, see for instance Corollary 13.2 in \cite{H}. The subspace $\ader L$ is stable under the adjoint action of $\der L$. Hence the statement. 
%  Let $D \in \der L$, then
%   the map
%$ $$  f_t:=\exp(Dt)\in \End(L)  $$
%  is an automorphism of $L$ for all $t\in \R$. For any $D'\in \der{L}$
%  we have
%  $$
%  \ds \frac{d}{dt}{\Big|_{t=0}}f_{t*}(D')=\frac{d}{dt}{\Big|_{t=0}}\bigl(\exp(Dt)D'\exp(-Dt)\bigr)=[D,D'].
 % $$
  %If  $D'\in \ader{L}$, then by Lemma~\ref{lem:autom},
%  $$
%    f_{t*}(D')\in \ader{L} \text{ for all } t\in \R.
%  $$
%  Therefore
%  $$
%  [D,D']=\ds \frac{d}{dt}{\Big|_{t=0}}f_{t*}(D')\in \ader{L}
%  $$
%  and so $\ader{L}$ is an ideal in $\der{L}$.
\end{proof}

\begin{remark}
This proves a conjecture from~\cite{BDV} which earlier was verified only in some special cases such as  nilpotent Lie algebras~\cite{KO}.
% \\[4pt]
% (2) This is true for all Lie algebras without assuming finite-dimensionality or that $\F=\R$ or $\C$. See Proposition~\ref{prop:ideal_super} below.
% \comm{\color{red} Provided we have a proof ...}
\end{remark}

\subsection{Almost inner derivations of Lie superalgebras}

\begin{definition}
  \label{def:superder}
 Let $L=L_\zer\oplus L_\one$ be a Lie superalgebra.
A homogeneous linear map $D: L\to L$ is called a \emph{derivation of degree } $\delta\in \Zt$, if $|D|=\delta$ and
\begin{equation}
  \label{eq:leibniz}
D([x,y])=[D(x),y]+(-1)^{\delta|x|}[x,D(y)],
\end{equation}
for all homogeneous $x\in L$.

  An arbitrary linear map $D\in \End(L)$ is called a derivation if its even and odd components $D_\zer\in \End(L)_\zer$ and $D_\one\in \End(L)_\one$ are
  derivations of degree $0$ and $1$, respectively.
\end{definition}

The super Jacobi identity implies that for every $a\in L$, the map
\[
  \ad_a: L\to L, \ x\mapsto [a,x]
\]
is a derivation. Such derivations are called \emph{inner}. It is straightforward to verify that the space
$\der{L}$ of derivations of $L$ is a subalgebra of the Lie superalgebra $\End(L)$ (with the bracket given by the commutator~\eqref{eq:commutator}) and that inner derivations $\ider{L}$ form an ideal in $\der{L}$.

\

The notion of outer derivations~\eqref{eq:outer}  and their cohomological description~\eqref{eq:cohom} carry over to the super case without modifications. However, the vanishing~\eqref{eq:vanish} does not hold in general even for simple Lie superalgebras and, as it turns out, they may have non-trivial outer derivations.

\begin{example}[\em The Euler derivation]
\label{ex:euler}
Every $\Z$-graded Lie superalgebra
$$ L=\bigoplus_{i\in \Z} L_i$$
has a special derivation $\cE$ called the \emph{Euler} (or \emph{grading}) derivation, defined
by
$$
\cE(x)=ix, \text{ if } x\in L_i.$$

\

For the Lie superalgebra $L=\gl(m|n)$, this derivation is inner, $\cE=\ad_H$, where $H$ is
the diagonal matrix
$$H=\mathrm{diag}(m+n,m+n-1,\ldots,1).$$

If $m\ne n$, then the restriction of the Euler derivation to the subsuperalgebra of traceless matrices
$$
\sl(m|n)=\{A\in \gl(m|n) \col \mathrm{str A=0}\}
$$
is still inner, since in this case we can replace $H$ by
$$
\tilde{H}=H-\frac{\mathrm{str} H}{m-n}I_{m+n} \in \sl(m|n),$$
and still get $\cE=\ad_{\tilde{H}}$.

However, for the Lie superalgebra $\sl(n|n)$, as well as for its simple quotient
$$\psl(n|n)=\sl(n|n)/\F I_{n|n},
$$
the Euler derivation is no longer inner (see Section~\ref{sec:classification}) and so $\oder{L}\ne 0$.

Let $W(n)$ be the Lie superalgebra of derivations of the (supercommutative) Grassmann algebra $$\Lambda_n=\F[\xi_1, \ldots, \xi_n]$$
with $n$ free odd generators (i.e.\ vector fields on the affine supermanifold $\F^{0|n})$.
The superalgebra $W(n)$ is simple (for $n\ge 3$) with a natural $\Z$-grading induced by the grading of $\Lambda_n$. The corresponding Euler derivation is inner:  $\cE=\ad_{\mathsf E}$, where $\mathsf E$ is the \emph{Euler vector field}
$$ \mathsf{E} = \sum_{i=1}^n \xi_i \frac{\partial}{\partial \xi_i}\in W(n).$$
However, the restriction of $\cE$ to the simple subsuperalgebra
$$S(n)\subset W(n)$$
of divergence-free vector fields is no longer inner and so $\oder{S(n)}\ne 0$.
\end{example}

\

We now introduce \ai derivations of Lie superalgebras.
Unlike derivations and inner derivations, the definition of \ai derivations for Lie superalgebras requires considering only homogeneous elements.

The following notion allows to define \ai derivations without explicitly involving elements.
\begin{definition}
An \emph{elementary} Lie superalgebra is a Lie superalgebra $E$ freely generated by one homogeneous element $x\in E$.
\end{definition}
If $x$ is even, then $E$ is just a one-dimensional abelian Lie algebra $\F x=\langle x \rangle$, and if $x$ is odd, then $E$ has dimension $(1|1)$ with $E_\one=\langle x \rangle$ and $E_\zer=\langle [x,x] \rangle$ .
Every homogeneous element $y\in L$ of a Lie superalgebra $L$ defines a canonical homomorphism
$$\varphi_y: E\to L$$
sending the generator $x$ to $y$.

\begin{definition}
  Let $L$  be a Lie superalgebra. A homogeneous derivation
$$
  D\in \der{L}_{\delta}
$$
  is called \emph{\ai}if for every homomorphism
$$
  \varphi: E\to L
$$
  from an elementary superalgebra $E$ there exists a homomorphism
$$
  \psi:E'\to L
$$
from an elementary superalgebra with the generator of parity $\delta$, such that the following diagram commutes
$$
    \begin{tikzcd} E'\times E\arrow["\varphi_D\times\varphi",rr]\arrow[dd,"\psi\times\varphi"]&&\der L\times L\arrow[dd]\\
      \\ L\times L\arrow["{[\cdot,\cdot]}",rr]&& L
      \end{tikzcd}
$$

An arbitrary derivation  $D\in \der{L}$  is called \ai if its even and odd components, $D_\zer\in \der{L}_\zer$ and  $D_\one\in \der{L}_\one$, are \ai.
  \end{definition}

We will denote by $\ader{L}$ the set of all \ai derivations of $L$.

\begin{remark}
  The above definition is equivalent to the requirement that $D$ satisfies the conditions~\eqref{eq:ai} and~\eqref{eq:ai2} of Definition~\ref{def:almost} for \emph{homogeneous} $x\in L$. These conditions may fail for non-homogeneous elements.
\end{remark}

By definition, $\ader{L}$ is a subsuperspace in $\der{L}$. Like in the purely even case, it is also a subalgebra.
\begin{proposition}
  Let $L$ be a Lie superalgebra. Then $\ader{L}$ is a Lie subsuperalgebra of $\der{L}$.
\end{proposition}
\begin{proof}
If $D$ and $D'$ are homogeneous \ai derivations of $L$, then the fact that $[D,D']\in \ader{L}$ is verified exactly as in the
the proof of Proposition~\ref{prop:subalg} with sign factors inserted in certain places when both $D$ and $D'$ are odd.
\end{proof}

\begin{remark}
 Since  the Lie algebra of the automorphism group $\aut L$ of a finite-dimensional Lie superalgebra $L$ over a field of characteristic zero
coincides with $\der L_\zer$, 
the same argument as in the proof of Proposition~\ref{prop:ideal} shows that the set of even \ai derivations $\ader{L}_\zer$ is an ideal in the Lie algebra $\der{L}_\zer$ of even derivations of $L$. Even though this proof cannot be directly extended to
whole subalgebra $\ader{L}\subset \der{L}$, we believe that $\ader{L}$ is an ideal in the Lie superalgebra $\der{L}$.
\end{remark}

\subsection{Examples of non-inner \ai derivations}

Naturally occurring Lie algebras typically do not have \ai derivations that are not inner. The smallest dimension of a Lie algebra with such a derivation over any field of characteristic $\ne 2$ is $5$ (see~\cite{BDV,Ve}) and the corresponding Lie algebras are nilpotent.

As we will show below, in the super case the situation is quite different. There are many examples of both even and odd non-inner \ai derivations in Lie superalgebras that are far from being solvable.

\begin{example}[\em An even non-inner \ai derivation]
\label{ex:even}
Let $L=L_\zer\oplus L_\one$ be a Lie superalgebra where
$$L_\zer=\sl(2) \ \text{ \ and \ } \ L_\one=V, \text{ the standard } \sl(2)-\text{module},$$
and the bracket is given by
$$
[(g,v),(g',v')]:=([g,g'],gv'-g'v).
$$
 Then the map
$$
  D: L\to L, \text{ given by } D|_{L_\zer}=0 \text{ and } D|_V=\Id_V,
$$
is an even \ai derivation which is not inner.
\end{example}
\begin{proof}
It is easy to see that $D$ is a derivation (for instance, because $L$ is a $\Z$-graded superalgebra and $D$ is its Euler derivation).

The derivation $D$ is not inner, since $I\not\in L_0=\sl(2)$, but it is \ai because
\[D(x)=
  \begin{cases}
    0=[0,x], & \text{ \ if \ } x\in L_\zer\\
    x=[g,x], & \text{ \ if \ } x\in L_\one \text{ and } g \text{ is any element of } \sl(2) \text { fixing } x.
\end{cases}
\]
\end{proof}

\

\begin{example}[\em An odd non-inner \ai derivation]
\label{ex:odd}
As in the previous example, let $V$ be the standard $\sl(2)$-module and consider a Lie superalgebra $L$ with
$$
  L_\zer=V \text{ \ and \ } L_\one=\sl(2)\oplus V,
$$
and the only nonzero component of the bracket being
$$
[\, ,]: L_\one\otimes L_\one \to L_\zer=V, \quad (g,w)\otimes (g',w')\mapsto gv'-g'v.
$$

Then the odd map  $ D: L\to L$,  given by
  $$
  D|_{L_\zer}=0 \text{ \ and \ } D|_{L_\one}=\mathrm{pr}_2: \sl(2)\oplus V \to V=L_0
  $$
  is an \ai derivation which is not inner.
\end{example}

\begin{proof}
The proof is analogous to the previous example and is left to the reader.
\end{proof}

\begin{remark}
An essential role in the above two examples played a Lie algebra $\g$ with a \emph{transitive representation} $V$ (that is, a $\g$-module with an element $x\in V$ such that $\g\cdot x=V$).
Such pairs $(\g,V)$ appear in the study of prehomogeneous vector spaces~\cite{Vi,SK}.
In the next subsection we show how they can be used to produced a large family of \ai derivations for Lie superalgebras.
\end{remark}

\subsection{Almost inner derivations and prehomogeneous vector spaces}

\begin{definition}
Recall that a \emph{prehomogeneous vector space}~\cite{SK} is a vector space $V$ with a linear action of an algebraic group $G$, such that $V$ has an open $G$-orbit. This is equivalent to the requirement that the associated representation on $V $of the Lie algebra $\g=\Lie(G)$ of $G$ is \emph{transitive}~\cite{Vi}, that is,
$$
\g\cdot v=V \text{ for some } v\in V.
$$
We will call such pairs $(G,V)$ and $(\g,V)$ \emph{prehomogeneous spaces}.
\end{definition}

% \begin{definition}
% Let $L=L_\zer\oplus L_\one$ be a Lie superalgebra and let $G$ be a Lie group such that $\Lie{G}=L_\zer$.
% We call $L$ a \emph{prehomogeneous} Lie superalgebra if $(G,L_\one)$ is a prehomogeneous vector space, i.e.\ if $L_\one$ has an open dense $G$-orbit.
% \end{definition}

\begin{proposition}
\label{prop:prehom}
Let $\g$ be a Lie algebra and let $V$ be an $\g$-module.
Consider the $\Z$-graded Lie superalgebra $L$ with the even part $L_\zer=\g$,
the odd part $L_\one=V$, and the bracket defined by
\[
  [x,y]=\begin{cases}
\    [x,y], & \text{if } x,y\in L_\zer,\\
\  xy, & \text{if }  x\in L_\zer, y\in L_\one,\\
\ 0, & \text{if }  x,y\in L_\one.
  \end{cases}
\]
Let $\cE\in \der{L}$ be the Euler derivation from Example~\ref{ex:euler}.

Suppose that $(\g,V)$ is a prehomogeneous space such that for every vector $w\in V$ its orbit $Gw\subset V $ is conical. Then $\cE$ is an \ai derivation of $L$.

If, in addition, $\g$ is perfect, that is, $\g=[\g,\g]$, then $\cE$ is a  non-inner \ai derivation.
\end{proposition}
\begin{proof}
  Indeed, let $\cO$ be an open $G$-orbit in $V$. Then for every $w\in \cO$ the tangent space $T_w\cO$ coincides with the whole $V$ and so we have
  $$\g w=T_w(\cO)=V.$$
  Now, consider $v\in V-\cO$. Since the orbit $Gv$ is conical, the line $\C^*v$ belongs to $Gv$ and therefore we obtain that
$$v \in T_v(\C^*v)\subset T_v(Gv)=\g v.$$
Thus $\cE(x)=x\in \g x$ for all $x\in V$ which shows that $\cE$ is an \ai derivation.

Now, if $\g$ is perfect, then for all $g\in \g$ the trace of $g$ in any finite-dimensional representation must be zero. Therefore, no $g\in \g$ can act on $V$ as non-zero scalar. Since $\cE$ acts as the identity on $L_\one=V$, this implies that the derivation $\cE$ is not inner.
\end{proof}

It is not hard to show that the conditions of the above proposition are also necessary, that is the derivation $\cE$ is almost inner only if $(\g,V)$ is a prehomogeneous space with conical orbits. We leave this to the reader.

\

Vinberg~\cite{Vi} described all prehomogeneous spaces $(\g,V)$  with a simple Lie algebra $\g$ and irreducible representation $V$ over $\C$. His result is as follows.
\begin{proposition}[\cite{Vi}]
  \label{prop:vi}
  The following table contains every prehomogeneous space $(\g,V)$, where $\g$ is a simple complex Lie algebra and $V$ is an irreducible $\g$-module.
  \begin{center}
  \begin{tabular}{|c|c|}
    \hline
    $\g$ & $V$ or $V^*$ \\
    \hline
    \hline
    $\sl(n), \ n\ge 2$ & $\C^n$\\
    \hline
    $\sl(2n+1), \ n\ge 2$ & $\Lambda^2(\C^{2n+1})$\\
    \hline
    $\sp(2n), \ n\ge 2$ & $\C^{2n}$\\
    \hline
    $\so(10)$ & $\mathrm{spin}(10)$\\
    \hline
  \end{tabular}
\end{center}
\end{proposition}

\

Sato and Kimura~\cite{SK} classified prehomogeneous spaces $(G,V)$ for complex reductive groups with an irreducible $V$ up to so-called \emph{castling equivalence}. In each equivalence there is a unique \emph{reduced} representative $(G,V)$ with minimal $\dim V$. The requirement that $\g$ is perfect restricts our attention only to semisimple groups $G$. Here is their list of reduced pairs $(\g,V)$ with a semisimple $\g$ and irreducible $\g$-module $V$.

\begin{proposition}[\cite{SK}]
 \label{prop:sk}
  Every prehomogeneous space $(\g,V)$ with a semisimple complex Lie algebra $\g$ and an irreducible $\g$-module $V$ is castling equivalent to one of the following reduced pairs.
  \begin{center}
    \begin{tabular}{|c|c|}
      \hline
      $\g$ & $V$ \\
      \hline
      \hline
      $\g\oplus \sl(n), \text{ where }$& {$\C^m\otimes \C^n$} \\
       $\g\subset \sl(m), \ n\ge m+1$ & \\
      \hline
      $\sl(m)\oplus \sl(n), \ 2n\le m$ & $\C^m\otimes \C^n$\\
      \hline
      $\sl(2m+1)\oplus \sl(n), \ n\le 2$ & $\Lambda^2(\C^{2m+1})\otimes \C^n$\\
      \hline
      $\sp(2m)\otimes \sl(n), \ n\le m$ & $\C^{2m}\otimes \C^n$\\
      \hline
      $\so(10)$ & $\mathrm{spin}(10)$\\
      \hline
    \end{tabular}
  \end{center}
\end{proposition}

In order to use these prehomogeneous spaces to construct \ai derivations of Proposition~\ref{prop:prehom}, we need to be able to decide when all orbits are conical.
 While the open orbit is always conical, this is a non-trivial task in general. However, since all prehomogeneous spaces in Vinberg's list from Proposition~\ref{prop:vi} have finitely many orbits (see~\cite{Ka2}), all their orbits are conical due to the following fact.

\begin{proposition}
\label{prop:conical}
Let $G$ be a complex algebraic group and let $V$ be a $G$-module with finitely many orbits. Then $(G,V)$ is a prehomogeneous space such that every $G$-orbit in $V$ is conical.
\end{proposition}
\begin{proof}
  Let
  $$V/G=\{\cO_1,\ldots,\cO_n\}
  $$
  be the set of $G$-orbits in $V$. The closure $\overline{\cO_i}$ of each orbit is an algebraic variety and,
  since $V$ is irreducible, the decomposition
  $$ V=\bigcup_{i=1}^n\overline{\cO_i}$$
  implies that $V=\overline{\cO_i}$ for some $i$. Thus, $(G,V)$ is a prehomogeneous space.

If $v$ and $v'$ belong to the same $G$-orbit, then
  $v'=gv$ for some $g\in G$ and, since $G$ acts linearly on $V$, for $\lambda\in \C^*$ we have   $\lambda v' = g\cdot (\lambda v)$. Therefore, $\C^*$ acts on the set of orbits $V/G$ and, since $|V/G|=n < \infty$, this gives a homomorphism $\C^*\to S_n$. But $\C^*$ is connected, therefore this homomorphism is constant and so every orbit is invariant under the $\C^*$-action, as required.
\end{proof}

Example~\ref{ex:odd} of an odd non-inner \ai derivation can also be generalized using prehomogeneous spaces.

\begin{proposition} Let $\g=\Lie(G)$ be the Lie algebra of a connected algebraic group $G$ and let $V$ be a $G$-module. Suppose that $\g$ is semisimple and that all $G$-orbits in $V$ are conical.

  Consider the Lie superalgebra $L$ with $L_\one=\g\oplus V$, \ $L_\zer=V$, and the bracket given by
  $$
  [L_\one,L_\one]=[L_\zer,L_\zer]=0
  $$
  and
  $$
[(X,w),v]=(0,Xv) \text{ \ for \ } (X,w)\in L_\one \text{ \ and \ } v\in L_\zer.$$
  Then the map $D\in \End{L}$ such that
 $D(v)=(0,v)$ for $v\in L_\zer$ and $D(L_\one)=0$, is an odd non-inner \ai derivation.
\end{proposition}
\begin{proof} We leave to the reader to check that $D$ is a derivation. There
  is no $X\in \g$ such that $Xv=v$ for all $v\in V$. Therefore $D$ is not inner.
  On the other hand, since all orbits of $G$ are conical, we have $\F v\in \g v$
  for any $v\neq 0$. Therefore for any $v\in V$, $v\neq 0$ there exists $X\in\g$ such that
  $Xv=v$. Thus, $D(v)=[(X,0),v]$. So $D$ is almost inner.
  \end{proof}

\begin{example}
Let $L$ denote a following subalgebra of $\gl(n|n+k)$ with $k<n$ consisting of block matrices
\[
 X= \left(     \begin{array}{c|cc}
      A & B & c \\
      \hline
      0 & A & d \\
      0 & 0 & C
    \end{array}   \right)
\]
where $A$ lies in some Lie subalgebra of $\gl(n)$, $C$ lies in some Lie subalgebra of $\gl(k)$, $B$ is an arbitrary traceless matrix, $c,d$ are arbitrary matrices in $\F^{nk}$.  Then the matrix
\[
  D= \left(     \begin{array}{c|cc}
    0 & 1_n & 0 \\
    \hline
    0 & 0 & 0 \\
    0 & 0 & 0
  \end{array}   \right)
\]
lies in the normalizer of $L$ and gives an example of an odd almost inner derivation. To check this, note that
$$
  [D,X]=
    D= \left(     \begin{array}{c|cc}
      0 & 0 & d \\
      \hline
      0 & 0 & 0 \\
      0 & 0 & 0
    \end{array}   \right)~.
% \left(\begin{matrix} 0&0&d\\0&0&0\\0&0&0\end{matrix}\right)
$$
Take $d\neq 0$ and choose $B$ such that $Bd=d$. Then
\[
  Y= \left(     \begin{array}{c|cc}
    0 & B & 0 \\
    \hline
    0 & 0 & 0 \\
    0 & 0 & 0
  \end{array}   \right)
\]
satisfies $[Y,X]=[D,X]$.
\end{example}

\

\section{Derivations of simple Lie superalgebras}
\label{sec:classification}

\subsection{Simple Lie superalgebras}

In this section we assume that $\F=\mathbb C$. All complex simple Lie superalgebras are classified by Kac~\cite{Ka}.
We combine the results of this classification in the following two tables. In the first table we list all $\Z$-graded Lie superalgebras
$\ds L=\bigoplus_{i=-1}^n L_i$,  such that $L_0$ is a reductive Lie algebra.
\renewcommand{\arraystretch}{1.5}
\begin{center}
		\begin{tabular}{|c|c|c|}
			\hline
			$L$ & $L_0$ & $L_i$\\
                  \hline
                  $\sl(m|n),\, m\ne n,\ mn\ge 2$&$\sl(m)\oplus\sl(n)\oplus \C$&$L_1=V_m\otimes V_n^*,\,L_{-1}=V_m^*\otimes V_n$ \\
                  \hline

                  $\psl(n|n),\, n\ge 2$&$\sl(n)\oplus\sl(n)$&$L_1=V_n\otimes V_n^*,\,L_{-1}=V_n^*\otimes V_n$ \\
                  \hline
                  $\mathfrak{spe}(n),\, n\ge 3$&$\sl(n)$&$L_1=S^2V_n,\,L_{-1}=\Lambda^2V_n$ \\
                   \hline
                    $\osp(2|2n),\, n\ge 2$&$\sp(2n)\oplus \C$&$L_1=V_n,\,L_{-1}=V_n$ \\
                  \hline
                    $W(n),\, n\ge 3$&$\gl(n)$&$L_i=\Lambda^iV_n\otimes V_n^*$ \\
                  \hline
                      $S(n),\, n\ge 3$&$\sl(n)$&$L_i=(\Lambda^iV_n\otimes V_n^*)/\Lambda^{i-1}V_n$ \\
                  \hline
                      $SH(n),\, n\ge 4$&$\so(n)$&$L_i=\Lambda^{i+2}V_n,\, i\leq n-3$  \\
                  \hline
                \end{tabular}
              \end{center}
          In this table we denote by $V_n$ the defining (standard) representation of the classical Lie algebras $\gl(n),\sl(n),\so(n)$ and $\sp(n)$.

              The simple Lie superalgebra  $S(n)$ for even $n$ has a simple deformation $\tilde S(n)$. It has a filtration such that
              $S(n)$ is the associated graded Lie superalgebra.

              The simple superalgebra $SH(n)$ is the derived superalgebra $[H(n),H(n)]$ of the Lie superalgebra $H(n)$ of Hamiltonian vector fields
              on the affine supermanifold $\F^{0|n}$ with the standard symplectic structure.

              All other simple complex Lie superalgebras have semisimple even part $L_\zer$, and $L_\one$ is a simple $L_\zer$-module.
              They are listed in the second table.

              \renewcommand{\arraystretch}{1.5}
\begin{center}
		\begin{tabular}{|c|c|c|}
			\hline
			$L$ & $L_\zer$ & $L_\one$\\
                   \hline
                   % $\osp(m|2n),\, m\neq 2$&$\so(m)\oplus\sp(2n)$&$V_m\otimes V_n$ \\
                   $\osp(m|2n),\, m\neq 2$&$\so(m)\oplus\sp(2n)$&$V_m\otimes V_{2n}$ \\
                   \hline
                    $D(2,1;a)$&$\sl(2)\oplus\sl(2)\oplus\sl(2)$&$V_2\otimes V_2\otimes V_2$ \\
                  \hline
                      $G(3)$&$\sl(2)\oplus G_2$&$V_2\otimes W_7$ \\
                  \hline
                      $F(4)$&$\sl(2)\oplus \so(7)$&$V_2\otimes\mathrm{spin}_7$  \\
                  \hline
                    $\psq(n),\ n\ge 3$&$\sl(n)$&$\ad$  \\
                  \hline
                \end{tabular}
              \end{center}
Here $W_7$ is the $7$-dimensional simple $G_2$-module.

% \subsection{Killing form and derivations}

% Simple Lie superalgebras with nondegenerate Killing form:
% \begin{itemize}
% \item $L$ is one of $A(m,n)$, with $m\ne n$, $B(m,n)$, $C(n)$, $D(m,n)$, with $m-n\ne 1$, $F(4)$, or $G_3$.
% \end{itemize}

%\subsection{Description of derivations}
% \begin{theorem}
% \label{thm:derivations}
% \begin{itemize}
% \item Superalgebras $A(m,n)$ with $m\ne n$, $B(m,n)$, $C(n)$, $D(m,n)$, $F(4)$, $G_3$, $W(n)$ or $S'(n)$ have no outer derivations.
% \item Superalgebras $A(n,n)$ with $n\ge 2$,  $P(n)$, and $S(n)$, have only one outer derivation: the grading operator
% \item $Q(n)$ has only one outer (odd) derivation sending $G_1$ to $G_0$ and $G_0$ to $0$.
% \item $H(n)$ (for $n\ge 5$) has two outer derivations: the grading and the gradient of the volume form $\ds \sum_i \partial_{\xi_i}(\xi_1\ldots\xi_n)\partial_{\xi_i}$
% \item $A(1,1)=psl(2,2)$ has an $sl(2)$ algebra of derivations:
%   the grading $h$ and $x,y$ interchanging $G_+$ and $G_-$.
% \end{itemize}
% \end{theorem}

\subsection{Description of derivations}
Here we describe the superalgebras of derivations for all simple complex Lie superalgebras. Although the theorem below was stated in several papers, starting with \cite{Ka}, we could not find a detailed proof in the literature. For this reason we include it here.
%for the reader convenience.
\begin{theorem}
\label{thm:derivations}
Let $L$ be a simple finite-dimensional Lie superalgebra over  $\C$.
\begin{enumerate}[(a)]
\item If $L$ is one of the Lie superalgebras
  $$\sl(m|n), \text{ with \ } m\ne n, \ \ \osp(m|2n), \ F(4), \ G(3), \ W(n), \text{  or \ } \tilde S(n),$$
  then $\oder{L}=0$ and so $\der{L}\simeq L$.
\item If $L$ is one of the superalgebras
 $$\psl(n|n), \text{ with \ } n\ge 2, \ \ \spe(n), \text{ \ with \ } n\ge 3, \text {\ or \ } S(n), \ \text{ \ with \ } n\ge 2,$$ then the space $\oder{L}$ is one-dimensional and is generated by the Euler derivation.
\item If $L=\psq(n)$, with $n\ge 3$, then the superalgebra $\oder{L}$,  has dimension  $(0|1)$ and is generated by the odd derivation sending $L_\one$ to $L_\zer$ and $L_\zer$ to $0$.
\item If $L=SH(n)$ with $n\ge 5$, then $\der{L}=\C  \opluslhrim  H(n)$.
\item If $L=\psl(2|2)$, then
  $$\der{L}\simeq D(2,1,-1)\simeq\sl(2) \opluslhrim  \psl(2|2)$$ and so $\oder{\psl(2|2)}\simeq \sl(2)$.
\end{enumerate}
\end{theorem}
\begin{proof}
  The proof will be presented in the sequence of Lemmas~\ref{center}---\ref{sl(2)}.

\begin{lemma}\label{center} Let $h\in L_\zer$ be a semisimple element which does not belong to the center of $ L_\zer$ and let  $D\in\der{L}$. If $D(h)=\lambda h$ then $\lambda=0$.
\end{lemma}

  \begin{proof} Consider the one-parameter subgroup $\exp (tD)$ in $\Aut(L)$. Since every automorphism preserves the spectrum
of $\ad_h$ the statement follows.
  \end{proof}
  \begin{lemma}\label{reduction}
Assume that a Levi subalgebra $L'\subset L_\zer$ has at most one-dimensional centralizer  $C(L')$ in $L$ and $C(L')$ is generated by a semisimple element.
Then
\begin{enumerate}[(i)]
\item $L_0=C(L')\oplus L'$ is a reductive subalgebra which acts semisimply on $L$.
  \item For every  $D\in\der L$ there exists
    $D'\equiv D\mod \ad{L}$ such that
    $D'(L_0)=0$.
  \item There exists an embedding of Lie superalgebras
    $\oder{L}\subset {\End}_{{L_0}}(L/L_0)/C(L')$.
    \end{enumerate}
\end{lemma}
\begin{proof} (i) is obvious.

  (ii) The restriction of $D$ to $L'$ is a cocycle in the complex computing $H^1(L';L)$. Since $L'$ is semisimple we have
  $H^1(L';L)$ and hence $D$ is a coboundary. So we can choose $D'$ such that $D'(L')=0$. This implies $D'(C(L'))\subset C(L') $. By our assumptions
  $C(L)=\F z$ for some semisimple $z\in L_\zer$.
By Lemma \ref{center} we have $D'(z)=0$ and therefore
$D'(L_0)=0$.

(iii) Let
$$\operatorname{Der}^{L_0}(L)=\{D\in \der L\mid D(L_0)=0\}.$$
For every $D\in \operatorname{Der}^{L_0}(L)$ and $y\in L_0$, $x\in L$ we have
$$D([y,x])=[y,D(x)].$$
Write $L=L_0\oplus M$ where $M$ is an $L_0$-submodule of $L$. The above identity implies that  $\operatorname{Der}^{L_0}(L)$ is a Lie subalgebra of ${\End}_{{L_0}}(M)$.
That induces the embedding
$$\oder{L}\subset {\End}_{{L_0}}(M)/(\ad L\cap \operatorname{Der}^{L_0}(L)).$$
Finally, the statement follows from the identity
$$\ad L\cap \operatorname{Der}^{L_0}(L)=C(L').$$
\end{proof}
\begin{corollary} Let $L$ be a simple Lie superalgebra such that $L_\zer$ is semisimple and $L_\one$ is a simple $L_\zer$-module.
  Then all even derivations of $L$ are inner. If
  $$\Hom_{L_\zer}(L_\one,L_\zer)=0$$
  then all derivations of $L$ are inner.
\end{corollary}
The above corollary implies Theorem~\ref{thm:derivations} for  $\osp(m|2n)$, $m\neq 2$, all exceptional superalgebras and $\mathfrak{psq}(n)$. In the latter case
$L_\one\simeq L_\zer$ as a $L_\zer$-module which gives an odd derivation.

\begin{lemma}\label{graded}
  Suppose that $L$ is $\Z$-graded with a grading
  $$L=\bigoplus_{i=-1}^nL_i$$
  such that
  \begin{enumerate}[(i)]
  \item the grading is compatible with the  parity, i.e.
    $$L_\zer=\bigoplus_{i\equiv \zer\!\!\pmod{2}} L_i \text{ \ and \ } L_\one=\bigoplus_{i\equiv \one\!\!\pmod{2}} L_i~;$$
%AV: $L_i\subset L_{\bar i}$;
    \item $L_0$ is a reductive Lie algebra satisfying the condition of Lemma~\ref{reduction};
    \item  $L_{-1}$ is a simple $L_0$-module;
      \item $\Hom_{L_0}(L_{-1},L)=\C J $, where $J$ is the inclusion $L_{-1}\hookrightarrow L$.
    \end{enumerate}
Then
\begin{enumerate}[(a)]
\item  If $L'=L_0$, then $\oder{L}=\C \cE$,
  where $\cE$   is the Euler derivation from Example~\ref{ex:euler}.
\item If $L'\neq L_0$, then all derivations of $L$ are inner.
\end{enumerate}
\end{lemma}
\begin{proof} Observe that $\cE$ is a derivation of $L$. Furthermore, since $L$ is simple we have $[L_{-1},x]=0$ if and only if $x\in L_{-1}$.
  Using Lemma \ref{reduction} it suffices to describe derivations $D$ such that $D(L_0)=0$.

Schur's lemma and the condition
  $\Hom_{L_0}(L_{-1},L)=\C $ imply that $D$ acts on $L_{-1}$ as a scalar operator $\lambda$. Let
  $$D':=D+\lambda\cE.$$
  We will prove that $D'=0$ by showing by induction on $i$ that $D'(L_i)=0$ for all $i$.

  The statement is trivial for $i=-1$.
  Let $x\in L_i$, $u\in L_{-1}$ then
  $$0=D'([u,x])=(-1)^{|D'|}[u,D'(x)].$$
  Therefore $D'(x)\in L_{-1}$, and by the condition $\Hom_{L_0}(L_{-1},L)=\C $, we get $D'(x)=0$. That implies (a), since $\cE\notin \ad L_0$.

  Now assume that $L'\neq L_0$. In this case $L_0$ has a non-trival center $\frak{z}$. But if  $z\in \frak{z}$, then $\ad_z$ acts by a scalar on $L_{-1}$ and annihilates $L_0$. Hence $\ad_z=\cE$ which proves (b).
\end{proof}
Note that Lemma \ref{graded} implies the theorem for all simple superalgebras such that $L_\one$ is not a simple $L_\zer$-module, except the three cases: $\tilde{S}(2n)$, $SH(n)$ and $\psl(2|2)$.
\begin{lemma}\label{Cartan}\ \
  \begin{itemize}
  \item{(a)} All derivations $\tilde{S}(2n)$ for $n\ge 2$ are inner.
  \item{(b)} $\der{SH(n)}=\C \cE \opluslhrim  H(n)$ for $n\ge 5$.
    \end{itemize}
\end{lemma}
\begin{proof} Let $L=\tilde S(2n)$. In this case we have a decomposition
$$
  L=L_{-1}\oplus\bigoplus_{i=0}^{2n-2} L,
$$
which satisfies $$[L_i,L_j]\subset L_{i+j} \text{ for } i\ge -1, j\ge 0,$$
while $$[L_{-1},L_{-1}]\subset L_{2n-2}.$$
Furthermore, the subalgebra $\ds \bigoplus_{i=0}^{2n-2} L$ is isomorphic to the similar subalgebra in $S(2n)$, in particular, $L_0=\sl(2n)$.
Each $L_i$ is a simple $L_0$-module and $L_i$ is not isomorphic to $L_j$ for $i\neq j$.
Therefore, if $D$ is a derivation such that $D(L_0)=0$, then there exist $$\lambda_{-1},\lambda_0,\lambda_1,\dots\lambda_{2n-2}\in\C$$
such that $\lambda_0=0$ and $D(x)=\lambda_ix$ for all $x\in L_i$. One can easily see that
$$
\lambda_i+\lambda_j=\lambda_{i+j}, \text{ for } i\ge -1,j\ge 0, \text{ and } 2\lambda_{-1}=\lambda_{2n-2}.
$$
Hence $\lambda_i=0$ for all $i$ which implies $D=0$. This proves that all derivations of $\tilde S(2n)$ are inner.

Now let $L=SH(n)$. In this case $L$ has a $\Z$-grading satisfying the conditions (1)---(3) of Lemma~\ref{graded}, with $L'=L_0\simeq\so(n)$ and each $L_i$ isomorphic to
$\Lambda^{i+2}(V)$, where $V$ is the standard representation of $L_0$. In particular, we have $$\Hom_{L_0}(L_{-1},L)=\C ^2.$$
Consider
 $D\in\der L$ such that $D(L_0)=0$. Then $D(L_{-1})\in L_{-1}\oplus L_{n-3}$. More precisely,
 $$
 D(\partial_i)=a\partial_i+b[L_\omega,\partial_i],\quad L_\omega =
 \sum_{i=1}^n \partial_i(\omega)\partial_i,\quad \omega=\xi_1\dots\xi_n,
 $$
 for some $a,b\in \C $.
 Let $D':=D+a\cE-b\ad_{L_\omega}$. Then $D'(L_{-1}\oplus L_{0})=0$ and, therefore, by the same argument as in the proof of Lemma~\ref{graded}, we obtain that
 $D'=0$. This proves that  $\der{SH(n)}=\C \cE \opluslhrim H(n)$.
\end{proof}

\begin{lemma}\label{sl(2)} For $L=\psl(2|2)$ we have $\der{L}\simeq D(2,1;-1)$.
\end{lemma}
\begin{proof} We first note that $L$ is an ideal in $D(2,1;-1)$ with
  $$\sl(2)\simeq D(2,1;-1)/\psl(2|2).$$
  On the other hand, using Lemma \ref{reduction}, we see that $\der{L}/L$ is a Lie subalgebra in $$\End_{L'}(L_{\one})=M_2(\C ).$$
 Since $\sl(2)\subset\der{L}/L$, it remains to show that $\der{L}/L$ does not contain a non-zero scalar matrix, which is straightforward.
  \end{proof}

  This concludes the proof of Theorem~\ref{thm:derivations}.
  \end{proof}
\section{Almost inner derivations of simple Lie superalgebras}
\label{sec:main}

\begin{theorem}
  \label{thm:main}
Every \ai derivation of a simple Lie superalgebra is inner.
\end{theorem}

\begin{proof} As follows from Theorem \ref{thm:derivations}, if not all derivations of $L$ are inner, then $L$ is one of the following superalgebras: $S(n), SH(n),\psq(n), \spe(n) $ or $\psl(n|n)$.

\

First we will deal with the outer derivations of the Euler type.
\begin{lemma}\label{grading-notinner}
 The Euler derivation $\cE$ in $S(n), SH(n),\spe(n) $ and $\psl(n|n)$ is not \ai.
  \end{lemma}
  \begin{proof} All these algebras admit a compatible $\Z$-grading as in Lemma \ref{graded} with
    a semisimple even part $L_0$.
%AV: simple  $L_0$.
%AV: The reader can check that there exists $i>0$ such that the corresponding algebraic group $G_0$ has no open orbit on $L_i$.
    In each case, there exists $i>0$ such that the algebraic group $G_0$ corresponding to $L_0$ has no open orbit in $L_i$.
    This implies the existence of a non-constant homogeneous $G_0$-invariant polynomial
    $$f\in\C [L_i]=S^\bullet(L_i^*).$$
    Take $x\in L_i$ such that
    $f(x)\neq 0$. Assume that $\cE$ is almost inner. Then there exists $y\in L$ such that $[y,x]=ix$. Since $L$ is $\Z$-graded, we may assume that $y\in L_0$.
   But then the $\C^*$-orbit
   $$\cal{O}_x=\{\exp(ty)\cdot x\col t\in \C^*\}$$
    is contained in the one-dimensional subspace $\C x\subset L_i$.
    Since $x,y\ne 0$, the closure of $\cal{O}_x$ coincides with $\C x\subset L_i$ and so it must contain $0$.
    But $f$ is invariant under $G_0$, therefore
    $$f(x)=f(0)=0$$
    and we obtain a contradiction.
  \end{proof}
  \begin{lemma}\label{notinner2}
    % Let $L$ have the grading $L=\bigoplus_{i=-1}^n L_i$ and $d$ be a derivation such $d(L_{-1})\subset L_n$ and $d(L_{-1})\neq 0$.
    % Then $d$ is not \ai.
    Let  $$L=\bigoplus_{i=-1}^n L_i$$
be a $\Z$-graded Lie superalgebra and let
$D\in \der{L}$ be a derivation such that
$$0\ne D(L_{-1})\subset L_n.$$
Then $D$ is not \ai.
  \end{lemma}
  \begin{proof}
    % Indeed, $[g_{-1},L]\cap L_n=0$.
Assume that $D$ is \ai and  choose $x\in L_{-1}$ such that $D(x)\ne 0$.
Then $D(x)=[a,x]$ for some $a\in L$. Since $D(x)\in L_n$ and
$[L,L_{-1}]\cap L_n=0$, this gives a contradiction.
  \end{proof}

The above lemmas cover all simple Lie superalgebras with nontrivial outer derivations except $\psq(n)$ and $\psl(2|2)$.

\

  Let $L=\psq(n)$ with $n\ge 3$ and $D$ be the odd derivation such that $D(L_\zer)=0$ and $D:L_\one\to L_\zer$ is an isomorphism.
  We claim that $D$ is not \ai.

  Indeed, identify $L_1$ with the set of traceless $n\times n$ matrices and choose
  $$x=\mathrm{diag}(\lambda_1,\dots,\lambda_n)\in L_1$$
  to be a diagonal matrix with entries $\lambda_1,\dots,\lambda_n$ satisfying the conditions
  $$
  \sum\lambda_i=\sum\lambda_i^{-1}=0, \text{ \ and \ } \lambda_i+\lambda_j\ne 0, \text{ for all } i,j.
  $$
  We will prove that there is no $y\in L_\one$ such that $[y,x]=D(x)$. In the matrix form, such $y$ should satisfy
  the equation
  $$xy+yx=x+a I_n \text{ \ for some \ } a\in\C.
  $$
  Since all off-diagonal entries of $xy+yx$ are zero and
  $\lambda_i+\lambda_j\neq 0$, we know that $y$ is also diagonal. Moreover, the diagonal elements of $y$ are
  $\ds \mu_i=\frac{a}{\lambda_i}+\frac{1}{2}$.
  On the other hand, $\mathrm{tr}(y)=0$. But $\ds\sum\mu_i=\frac{n}{2}$. We obtain a contradiction.

\

  Finally, let us consider the case $L=\psl(2|2)$.
  Note that  $\der{L}=\sl(2) \opluslhrim  L$.
  Since \ai derivations form an ideal in $\der L$, it suffices to check
  that $\sl(2)$ contains a derivation $D\notin \ader L$.  One can take the Euler derivation $D=\cE$  and use Lemma~\ref{grading-notinner}.

  \end{proof}

  \section{Almost inner derivations of quasireductive superalgebras}
\label{sec:quasi}
Recall that a superalgebra $L$ is called \emph{quasireductive} if $L_0$ is a reductive Lie algebra and $L_\one$ is a semisimple $L_\zer$-module (see~\cite{Se}).
We assume that $\F$ is an algebraically closed field of characteristic $0$. By $G$ we denote a connected reductive algebraic group with Lie algebra $L_\zer$. Assume that $G$ acts on $L$ by adjoint representation.

  Denote by $\dero L$ the subalgebra of $\der L$ consisting of all $D$ such that $D(L_\zer)=0$. For any subalgebra $K\subset \der L$ denote by $K^\circ$ the intersection $\dero L\cap K$.
  \begin{lemma}\label{lem-reduction}  Let $L$ be quasireductive. Then
    $\ader L/\ider L$ is a subalgebra in \\ $\adero L/\idero L$.
  \end{lemma}
  \begin{proof} Let $D\in\ader L_\zer$. Then $D(z)=0$ for any central element $z\in L_\zer$. On the other, $[L_\zer,L_\zer]$ is semisimple and therefore all
    derivations of $[L_\zer,L_\zer]$ are inner. Therefore there exists $y\in L_\zer$ such that $(D-\ad_y)(L_\zer)=0$.

    Now let $D\in\ader L_\one$. Then $H^1([L_\zer,L_\zer],L_\one)=0$ and therefore we may choose $D$ such that $D([L_\zer,L_\zer])=0$.
    Thus, $D([x,y])=[x,D(y)]$ for any $y\in L$ and $x\in [L_\zer,L_\zer]$.
    Therefore if $Z$ is the center of $L_\zer$ then $D(Z)\subset L_\one^{[L_\zer,L_\zer]}$. On the other hand, we have $H^1(Z,M)=H^1(Z,M^Z)$ for any semisimple $Z$-module $M$. Thus,
    there exists $u\in L_\one^{[L_\zer,L_\zer]}$ such that $(D-\ad_u)(L_\zer)\subset L_\one^{L_\zer}$. But then $(D-\ad_u)(L_\zer)=0$ since $D-\ad_u$ is almost inner.
  \end{proof}
  \begin{corollary}\label{oddideal} If $L$ is quasireductive, then  $(\ader L/\ider L)_\one$ is an abelian ideal in  $\ader L/\ider L$.
  \end{corollary}
  \begin{proof} Follows from above Lemma since $[\dero L_\one,\dero L_\one]=0$.
  \end{proof}
  \begin{example}
    \label{ex-simple}
We revisit the example of Proposition  \ref{prop:prehom}. Let $L$ be a quasireductive superalgebra such that $[L_\one,L_\one]=0$ and $L_\one$ is a simple faithful  $L_\zer$-module. Then
    $$\dero L=\End_{L_\zer}(L_\one)\oplus\Hom_{L_\zer}(L_\one,L_\zer)=\F\mathcal E\oplus \Hom_{L_\zer}(L_\one,L_\zer).$$
    The second equality is a consequence of Schur's lemma.

    Note that $\adero L_\one=0$ since $[L_\one,L_\one]=0$.
    If $L_\zer$ has non-trivial center then $\mathcal E=\ad_z$ for some central element $z$.
    Assume now that $L_\zer$ is semisimple. Then $\mathcal E$ is almost inner if and only if all $G$-orbits in $L_\one$ are conical.
    In this case $\ader L/\ider L=\F\mathcal E$.
    \end{example}
  \begin{lemma}\label{lem-evenpart} If $L$ is quasireductive then $\adero L_\zer$ is an abelian Lie algebra which acts semisimply on $L$.
  \end{lemma}
  \begin{proof} Let $D\in \adero L_\zer$. For any simple $L_\zer$-submodule $W\subset L_\one$ we have $D(W)\subset W$ and since $D$ commutes with the action of $L_\zer$,
    it acts on $W$ by a scalar by Schur's lemma.
    Consider the decomposition
    \begin{equation}\label{decomposition}
      L_\one=\bigoplus_{i=1}^k V_i^{m_i},
      \end{equation}
    where $V_1,\dots,V_k$ are pairwise non-isomoprhic simple $L_\zer$-modules. Then $D|_{V_i^{m_i}}=\lambda_i\operatorname{id}$ for some $\lambda_1,\dots\lambda_k\in\F$. This implies the statement.
  \end{proof}
  \begin{remark}\label{remark-conical} Note that  if $D\in \adero L_\zer$ acts by non-zero scalar on $V_i^{m_i}$ then all orbits of $G$ in $V_i^{m_i}$
    must be conical.
    \end{remark}
  \begin{proposition}\label{abelian} If $L$ is quasireductive then $\adero L$ and hence  $\ader L/\ider L$ is abelian.
  \end{proposition}
  \begin{proof} We will prove that $\adero L$ is abelian. Let $D_0\in\adero L_\zer$ and $D_1\in\adero L_\one$. Consider the decomposition (\ref{decomposition}).
    Then $D_1(V_i)\neq 0$ if $V_i$ is isomorphic to an irreducible submodule of $L_\zer$. In other words, either $V_i$ is trivial or is isomorphic to some simple ideal of $L_\zer$. On the other hand, in this case not all $G$-orbits in $V_i$ are conical and hence $D_0(V_i)=0$ by Remark \ref{remark-conical}. Therefore $[D_0,D_1]=0$.
  \end{proof}
  \begin{conjecture} If $L$ is quasireductive then any almost inner derivation which is not inner is even.
  \end{conjecture}
  The proposition below confirms above conjecture for quasireductive superlagebras with trivial center.
  \begin{proposition} Assume that $L$ is quasireductive superalgebra with trivial center. Then $\ader L_\one=\ider L_\one$.
  \end{proposition}
  \begin{proof} We will use Lemma 5.6 and Theorem 6.9 of~\cite{Se}. Denote by $C(L)$ the sum of minimal ideals of $L$.  Lemma 5.6 implies that
    \begin{equation}\label{reductive}
      S(L)=M\oplus \bigoplus_{i=1}^k S_i\oplus \bigoplus_{j=1}^l T_j,
      \end{equation}
    where each $S_i$ is a simple Lie superalgebra and $T_j=\mathfrak{k}_j\otimes \F[\theta_j]$ with simple Lie algebra $\mathfrak{k}_j$ and odd $\theta_j$.
    Theorem 6.9 implies that $L=R(L) \opluslhrim C(L)$ where $R(L)$ is a a quasireductive Lie superalgebra with abelian $R(L)_\one$.
    Furthermore, $R(L)_\one$ is a subalgebra in $\operatorname{Der}^\circ(C(L))_\one$.  From Section 3 we know that $\operatorname{Der}^\circ(S_i)_\one=0$ unless
    $S_i$ is isomorphic to $\mathfrak{psq}(n)$ for $n\ge 3$ and
    $\operatorname{Der}^\circ(S_i)_\one$ is one-dimensional if $S_i\simeq\mathfrak{psq}(n)$. A simple computation shows that
    $$\operatorname{Der}^\circ(T_j)=\F\frac{\partial}{\partial{\theta_j}}.$$
    If $D\in\dero L_\one$ then $D$ is a linear combination of $\ds\frac{\partial}{\partial{\theta_j}}$ for $j=1,\dots l$ and
      $u_i\in\operatorname{Der}^\circ(S_i)_\one$ for $i$ such that $S_i\simeq \mathfrak{psl}(n)$. Not that none of $\theta_j$ or $u_i$ is an almost
      inner derivation of $C(L)$. Thus, if
      $$D=\sum c_j\theta_j+\sum d_iu_i$$ is an almost inner derivation of $L$ then $D\in R(L)$ and thus $D$ is inner.
    \end{proof}

    \begin{remark}\label{towclass}
  Observe that in the decomposition (\ref{reductive})  the almost inner derivations of $T_j$ and $S_i$ coincide with inner derivations. Thus,
      $$ \ader L/\ider L=\operatorname{aDer}(L')/\operatorname{iDer}(L') $$
where
$$
L'=L/\left( \bigoplus_{i=1}^k S_i\oplus \bigoplus_{j=1}^l T_j\right).
$$
      On the other hand, $L'=L'_\zer\oplus L_\one'$ with $[L_\one',L'_\one]=0$. In view of Example~\ref{ex-simple} and Remark~\ref{remark-conical} we can see
      that classification of almost inner derivations of centerless quasireductive superalgebras can be reduced to classification of
      representations of semisimple groups with conical orbits.
    \end{remark}
    \begin{proposition}\label{ext} Let $\mathcal G$ be a connected quasireductive supergroup with Lie superalgebra $L$. Assume that $L$ has trivial center.
      $\aaut {L}\subset \aut L$ be a supergoup with Lie superalgebra $\ader L$. Then $\aaut L$ is a maximal connected subgroup of $\aut L$ such that
      $$\operatorname{Ext}^\cdot_{\aaut{L}}(\F,\F)\simeq \operatorname{Ext}^\cdot_{\mathcal G}(\F,\F)$$ is an isomorphism of algebras.
    \end{proposition}
    \begin{proof} Follows from the identity, \cite{BKN}
      $$ \operatorname{Ext}^\cdot_{\mathcal G}(\F,\F)=H^\cdot(L;L_\zer,\F)=S^\cdot(L_\one^*)^G.$$
      \end{proof}
    \section {More examples and open questions}
    \label{sec:final}
     Note that Example~\ref{ex:odd} can be modified to obtain an example of non-solvable Lie algebra with almost inner but not inner derivation.
      Let $L$ denote the following subalgebra of $\gl(2n+k)$ with $k<n$ consisting of block matrices
\[
 X= \left(     \begin{matrix}
      0 & B & c \\
      0 & 0 & d \\
      0 & 0 & C
    \end{matrix}   \right)
\]
where $C$ lies in some Lie subalgebra of $\gl(k)$, $B$ is an arbitrary traceless matrix, and $c,d$ are arbitrary $n\times k $ matrices.  Then $\ad_D$ for
\[
 D= \left(     \begin{matrix}
      0 & 1_n & 0 \\
      0 & 0 & 0 \\
      0 & 0 & 0
    \end{matrix}   \right)
\]
gives an almost inner derivation of $L$ which is not inner.

The group version of the above examples gives an example of an almost inner automorphism for groups.
If $\F$ is a finite field this gives an example of almost inner but not inner automorphism of a finite group.

We know that for a finite group $G$ and algebraically closed field $\F$ of characteristic $0$ the group $\aaut(G)$ is exactly the set of automorphisms $\varphi$ such that $V^\varphi\simeq V$ for any $\F[G]$-module $V$.

Note that if
$\mathcal G$ is a quasireductive supergroup then we can have an automorphism $\varphi$ which is not almost inner but still $V^\varphi\simeq V$ for any $\mathcal G$-module $V$. Indeed, consider $\mathcal G=\R^{0|1}$ and an
automorphism $\varphi$ acting by a scalar operator. It would be interesting to characterize almost inner automorphisms in terms of representations of $\mathcal G$.


\begin{thebibliography}{10}

\bibitem[AS]{AS} A. Amiri, F. Saeedi, On pointwise inner derivations of Lie algebras. Asian-Eur. J. Math. 11 (2018):5, 1850070, 7 pp.

  \bibitem [BKN] {BKN} B.~Boe, J.~Kujawa, and D.~Nakano, \emph{Cohomology and support varieties for Lie superalgebras}.
		Transactions of the American Mathematical Society, 362.12 (2010): 6551-6590.

\bibitem[BDV]{BDV} D. Burde, K. Dekimpe, B. Verbeke, Almost inner derivations of Lie algebras. J. Algebra Appl. 17 (2018):11, p.1850214 1--26.

\bibitem[BDV2]{BDV2}   D. Burde, K. Dekimpe, B. Verbeke, Almost inner derivations of Lie algebras II.
  Internat. J. Algebra Comput. 31 (2021):2, 341--364.

\bibitem[Bu]{Bu}  W. Burnside, On the outer automorphisms of a group. Proc. London Math. Soc. 11 (1913), 40--42.


% \bibitem[DM]{DM} P. Deligne, J. Morgan, Notes on supersymmetry (following Joseph Bernstein). Quantum Fields and Strings: a Course for Mathematicians, v.I  (1999), 41--97.

\bibitem[DG]{DG} H. Dietrich, W.A. de Graaf, A computational approach to almost-inner derivations.
   J. Symbolic Comput. 125 (2024), 102312, 9 pp.

\bibitem[FHS]{FHS} Y. Fang, G. Han, B. Sun,   Conjugacy and element-conjugacy of homomorphisms of compact Lie groups.
  Pacific J. Math. 283 (2016):1, 75--83.


\bibitem[Go]{Go}  V. Goksel,  A local-global conjugacy question arising from arithmetic dynamics.
  Groups Geom. Dyn. 18 (2024):2, 635--648.

\bibitem[GW]{GW} G.S. Gordon, E.N. Wilson,   Isospectral deformations of compact solvmanifolds. J. Diff. Geom. 19 (1984),  214--256.
  
\bibitem[H]{H} J. Humphreys, Linear algebraic groups, Springer-Verlag, 1975.
  
\bibitem[Ka]{Ka} V.G. Kac, Lie superalgebras.  Advances in Math.  26 (1977):1, 8--96.

\bibitem[Ka2]{Ka2} V.G. Kac, Some remarks on nilpotent orbits, J. Algebra 64 (1980), 190--213.

\bibitem[Ku]{Ku}  B. Kunyavskii, Local-global invariants of finite and infinite groups: around Burnside from another side.
  Expo. Math. 31 (2013), 256--273.

\bibitem[KO]{KO} B. Kunyavskii,  V.Z. Ostapenko, Tate--Shafarevich groups and algebras. Intern. J. Algebra Comput. 33 (2023), 819--836.

 \bibitem[La]{La}  M. Larsen, On the conjugacy of element-conjugate homomorphisms. Israel J. Math. 88  (1994), 253--277

 \bibitem[O]{O} T. Ono, A note on Shafarevich--Tate sets for finite groups. Proc. Japan Acad. Ser. A Math. Sci. 74 (1998):5, 77--79.

 \bibitem[SK]{SK}  M. Sato, T. Kimura, A classiﬁcation of irreducible prehomogeneous vector spaces and their relative invariants, Nagoya Math. J. 65 (1977), 1--155.

 \bibitem[Se]{Se} V.~Serganova, Quasireductive supergroups. New developments in Lie theory and its applications, 544 (2011), 141--159.


 \bibitem[SSB]{SSB} S. Sheikh-Mohseni, F. Saeedi, A.M. Badrkhani,
   On special subalgebras of derivations of Lie algebras.    Asian-Eur. J. Math. 8 (2015):2, 1550032, 12 pp.

 \bibitem[Ve]{Ve}
   B.\,Verbeke, Almost inner derivations of Lie algebras. PhD\ thesis. KU Leuven, 2020.
\    \url{https://homepage.univie.ac.at/dietrich.burde/papers/thesis_verbeke_2020.pdf}

 \bibitem[Vi]{Vi} E.B.~Vinberg, Invariant linear connections in a homogeneous spaces. Tr. Mosk. Mat. Ob. 9 (1960),  191--210.

 \bibitem[Wa]{Wa} S. Wang,  On local and global conjugacy.
   J. Algebra 439 (2015), 334–359.

 \bibitem[We]{We}  M. Weidner,  Pseudocharacters of homomorphisms into classical groups.    Transform. Groups 25 (2020):4, 1345--1370.

 \bibitem[Ya]{Ya}  M.K. Yadav, Class preserving automorphisms of finite p-groups: a survey. “Groups St Andrews 2009 in Bath. v.2”, London Math. Soc. Lecture Note Ser., 388 (2011), 569--579.

  % \bibitem[Sch]{Sch} Scheunert, Lie superalgebras

 % \bibitem[Wa2]{Wa2} S. Wang,
 %   On dimension data and local vs. global conjugacy. Fifth Internat. Congress of Chinese Mathem.   AMS/IP Stud. Adv. Math., 51 (2012), pt. 1, 365--382.

 % \bibitem[Wa3]{Wa3} Song Wang, Multiplicity one, local and global conjugacy. Sci. China Math. 63 (2020), no. 6, 1029–1038.

 % \bibitem[Yu]{Yu} Jun Yu,
 %   On the dimension datum of a subgroup.  Duke Math. J. 165 (2016):14, 2683--2736.

 % \bibitem[AYY]{AYY}  J.An, J.-K. Yu, Jun Yu,
 %   On the dimension datum of a subgroup and its application to isospectral manifolds.
 %   J. Differential Geom. 94 (2013):1, 59--85.

\end{thebibliography}
\end{document}